\documentclass[a4paper,reqno,12pt]{amsart}
\usepackage[latin1]{inputenc}
\usepackage[T1]{fontenc}
\usepackage[english]{babel}
\usepackage{appendix}
\usepackage[english]{minitoc}
\usepackage[nice]{nicefrac}
\usepackage{mathpazo,xcolor} \usepackage[pdftex]{hyperref}
\usepackage{marginnote}
\newtheorem{theorem}{Theorem}[section]   
\newtheorem{prop}[theorem]{Proposition} \theoremstyle{definition}  \newtheorem{defn}[theorem]{Definition}
 \newtheorem{remark}[theorem]{Remark} \newtheorem{example}[theorem]{Example} 
\newtheorem{cexample}[theorem]{Counterexample} 
 \numberwithin{equation}{section}
\usepackage[top=2.4cm,bottom=2.4cm, left=2.2cm,right=2.2cm]{geometry}
\usepackage[framemethod=tikz]{mdframed}
\usepackage{lipsum}

\definecolor{mycolor}{rgb}{0.122, 0.835, 0.998}
\newmdenv[innerlinewidth=0.5pt, roundcorner=4pt,linecolor=mycolor,innerleftmargin=6pt,
innerrightmargin=6pt,innertopmargin=6pt,innerbottommargin=6pt]{mybox}
\newcommand{\norm}[1]{\left\Vert#1\right\Vert} \newcommand{\scal}[1]{\left<#1\right>}
  \newcommand{\R}{\mathbb{R}} \newcommand{\Z}{\mathbb{Z}} \newcommand{\C}{\mathbb{C}}
  \newcommand{\D}{\mathbb{D}} \newcommand{\BC}{\mathbb{T}} 
   
 \newcommand{\eu}{{e_+}}  \newcommand{\es}{{e_-}}
\newcommand{\Hbc}{\mathcal{H}^{2,\nu}(\BC)}

\begin{document}

\title[]{Bicomplex frames} 
\author{Aiad El Gourari}
\address{(A.E.) Department of Mathematics, Faculty of Sciences,
\newline
 Ibn Tofa\"il University, Kenitra}
\email{aiadelgourari@gmail.com}

\author{Allal Ghanmi}
\address{(A.G.) Analysis, P.D.G $\&$ Spectral Geometry. Lab. M.I.A.-S.I., CeReMAR, 
\newline Department of Mathematics, P.O. Box 1014,  Faculty of Sciences,
\newline Mohammed V University in Rabat, Morocco}
\email{ag@fsr.ac.ma}

\author{Mohammed Souid El Ainin}
\address{(M.S.) Faculty of Law, Economics and Social Sciences,
	\newline Ibn Zohr University, Agadir, Morocco}
\email{msouidelainin@yahoo.fr}

\begin{abstract}
The main purpose is to introduce the so-called bicomplex (bc)-frames which is a special extension to bicomplex infinite Hilbert spaces of the classical frames. 
The crucial result is the characterization of bc-frames in terms of their idempotent components, giving rise to  generalization of certain results to bc-frames.
 Although the extension is natural, many basic properties satisfied by classical frames do not remain valid for bc-frames, unless we restrict ourself to complex-valued Hilbert space on bicomplex numbers. By benefiting from insight provided by the classical frame theory, we discuss the construction of bc-frame operator and Weyl--Heisenberg bc-frames and we provide some new ones which are appropriate for bc-frames.
\end{abstract}



\maketitle

\section{Introduction} \label{s1}

 Frames are generalizations of orthonormal basis and can be defined as "sets" of vectors (not necessary independent) giving the explicit expansion of any arbitrary vector in the space  as a linear combination of the elements in the frame. 
They were introduced by Duffin and Schaeffer, in early fifties, in the framework of nonharmonic Fourier series (see \cite{DuffinSchaeffer1952}). 
 It is only in 1986 that a landmark development was given by Daubechies, Grossmann and Meyer \cite{DaubechiesGrassmannMeyer1986}. Since then,  they have been extensively studied in different branches of mathematics and engineering sciences, such as signal analysis 
 and especially in connection with sampling theorems \cite{Benedetto1993}, as well as in image processing,
 quantum information,  
  filter bank theory,  
 robust transmission, 
  coding and communications \cite{BolcskeiHlawatschFeichtinger1998,Daubechies1992,DonohoElad2003,EldarForney2002,Heil2010,Gavrut2006,Kovacevic2008}.
   For further details and tools on frames one can refer to the very nicely written surveys and research tutorials 
  \cite{Casazza2000,CasazzaKutyniok2012,CasazzaLynch2016,Christensen2016,Gr\"ochenig2001,HeilWalnut1989}. 
  
 The aim of the present paper is to introduce and study bicomplex (bc) frames for infinite bicomplex Hilbert spaces incorporating classical ones for complex Hilbert spaces.
 The motivation of considering the bicomplex setting lies in the fact that this model 
 can serves to represent color image encoding in image processing.  
  Example \ref{Fexple2} and the results we establish in Section 3 will illustrate further the motivation of considering bc-frames.
The main feature of bc-frames is Theorem \ref{thmchar} concerning the problem of characterizing these bc-frames in terms of the standard ones for appropriate Hilbert spaces. Accordingly, important basic and elementary  properties for bc-frames are described forthwith thanks to this special characterization.
Special attention will also be given to the problem of defining bicomplex frame operators and bicomplex Weyl--Heisenberg frames.

This paper is organized as follows. In Section 2, we review the structure of the infinite bicomplex Hilbert space including $L^{2}(\BC,e^{-\sigma |Z|^2}d\lambda)$; $\sigma \geq 0$. In Section 3, we introduce some basic notions and we state and prove our main results related to bc-frames. 
 Section 4 begins with a brief discussion of associated bicomplex frame operator. The rest of this section deals with Weyl--Heisenberg bc-frames.

In next sections $\Z$, $\R$, $\C$ and $\BC$ will respectively denote the integers, the real, the complex and the bicomplex numbers.

 \section{Preliminaries: Infinite bicomplex Hilbert spaces}

This section is a brief review of needed notions and results from the theory of infinite bicomplex Hilbert spaces. For a lot of the material and background we refer the reader to \cite{Price1991,RochonShapiro2004,RochonTremmblay2004,RochonTremmblay2006}. 
To start, recall first that the bicomplex numbers are special generalization of complex numbers. In fact, they are complex numbers $Z= z_1 + j z_2$ with complex coefficients $z_1,z_2 \in \C=\C_i$, 
where $j$ is a pure imaginary unit independent of $i$ such that $ij=ji$.
This defines a commutative (non division) algebra over $\C$ where the addition and multiplication operation are defined in a natural way. Conjugates of given $Z= z_1 + j z_2\in \BC$, with respect to $i$, $j$ and $ij$,  are defined by 
$\widetilde{Z}= \overline{z_1} + j \overline{z_2}$, $Z^\dag= z_1 - j z_2$ and $Z^*= \overline{z_1} - j \overline{z_2}$, respectively.
It should be mentioned here that the nullity of $ ZZ^\dag =z_1^2+z_2^2$, which 
is equivalent to $Z= \lambda (1\pm ij )$ for certain complex number $\lambda\in \C$, characterizes those that are zero divisors in $\BC$, while   $ZZ^{\dag}\ne 0$ characterizes those that are invertible. 
Thus by considering the idempotent elements
$$ \eu  = \frac{1+ij}2 \quad \mbox{and} \quad \es  = \frac{1-ij}2$$
we have the identities
$ \eu ^2=\eu $ $\es ^2=\es $, $\eu +\es  =1 $, $\eu -\es  =ij$ and $\eu \es =0$.
Therefore, for any $Z= z_1 + j z_2\in \BC$ there exist unique complex numbers $\alpha,\beta$, such that
\begin{align}\label{ib}
Z= (z_1 - i z_2) \eu  +  (z_1 + i z_2) \es  = \alpha \eu  +  \beta \es.
\end{align}
Here $\alpha=z_1 - i z_2,\beta=z_1 + i z_2\in$.
The idempotent representation \eqref{ib} is crucial and simplifies considerably the computation with bicomplex numbers. In particular, the different 
$Z^\dag$-, $\widetilde{Z}$- and $Z^*$-conjugates read simply 
 $$Z^\dag=\beta \eu  + \alpha \es , \quad \widetilde{Z}=\overline{\beta}\eu  + \overline{\alpha}\es  \quad\mbox{and} \quad Z^*= \overline{\alpha}\eu  + \overline{\beta}\es .$$

Infinite bicomplex Hilbert space is defined by means of a special extension of 
the notions of inner product and norm to the $\BC$-modules. More generally, if  $M$ is a $\BC$-module, we consider the $\C$-vector spaces $V^+=M\eu$ and $V^- = M \es$, so that one can see $M$ as the $\C$-vector space $M'=V^+\oplus V^-$. In general, $V^+$ and $V^-$ bear no structural similarities.
Accordingly, an inner product on $M$ is a given functional 
$$\scal{ \cdot, \cdot } : M\times M \longrightarrow \BC$$ satisfying
\begin{enumerate}
	\item $\scal{\phi, \tau\psi + \varphi }  = \tau^{*} \scal{\phi,\psi} + \scal{\phi,\varphi }$
for every $\tau\in \BC$ and $\phi, \psi , \varphi \in M$, and
\item $ \scal{\phi, \psi } = \scal{\psi, \phi }^*$
as well as 
\item $ \scal{\phi, \phi}=0$ if and only if $\phi=0$.
\end{enumerate}
Therefore, the projection $\scal{ \cdot, \cdot }_{V^\pm}$ of $\scal{ \cdot, \cdot }$ to  $V^\pm$ is a standard scalar product on $V^\pm$. Indeed, we have
$$
\scal{ \phi, \varphi } =  \scal{\phi^+, \varphi^+}_{V^+} \eu   + \scal{\phi^-, \varphi^-}_{V^-} \es  ,
$$
where $\phi, \varphi$ belong to $M $ identified to $M'=V^+ \oplus V^-$ and $\varphi^\pm:= \varphi e_\pm \in V^\pm$ and  $\psi^\pm:= \psi e_\pm \in V^\pm$.
It should be noted here that any $\BC$-scaler product on $M$ is completely determined as described above (see \cite[Theorem 2.6]{GLMRochon2010AFA}).

A bicomplex norm on a given $\BC$-module $M$ is the data of a map $\norm{\cdot}: M\longrightarrow \mathbb{R}$ satisfying
\begin{enumerate}
	\item $\norm{\cdot}$ is a norm on the vector space $V^+ \oplus V^-$, and
	\item $\norm{ \lambda \phi } \leq \sqrt{2} |\lambda| \norm{\phi}$ for all $\lambda \in \BC$ and all $\phi \in  M$.
\end{enumerate}
In this case $(M,\norm{\cdot})$ is called a normed bicomplex-module.
As for normed $\C$-vector spaces, a bicomplex norm can always be induced from a $\BC$-scalar product by considering
\begin{align}\label{inducednorm}
\norm{\phi}^2 =  \frac{1}{2} \left(  \scal{\phi^+, \phi^+}_{V^+} + \scal{\phi^-, \phi^-}_{V^-} \right)     = \left| \scal{\phi, \phi}\right|,
\end{align}
where $\phi=\phi^+ + \phi^- $ thanks to the identification of $M$ to $V^+ \oplus V^-$. The modulus $|\cdot|$ denotes the usual Euclidean norm in $\mathbb{R}^{4}$.
 The norm in \eqref{inducednorm} obeys a generalized Schwarz inequality (\cite[Theorem 3.7]{GLMRochon2010AFA})
\begin{align}\label{SchIneq}
\left|\scal{\phi, \varphi}\right| \leq \sqrt{2} \norm{\phi} \norm{\varphi} .
\end{align}

Accordingly, one defines an infinite bicomplex Hilbert space to be a $\BC$-inner product space $(M, \scal{\cdot, \cdot})$. This is complete with respect to the induced $\BC$-norm \eqref{inducednorm} which is equivalent to 	$(V^\pm, \scal{\cdot, \cdot}_{V^\pm})$ be $\C$-Hilbert spaces. This characterization  is contained in Theorems 3.4, 3.5 and Corollary 3.6 of \cite{GLMRochon2010AFA}.
As example of infinite bicomplex Hilbert space, one consider the one associated to 
the trivial bicomplex inner product
\begin{align}\label{spBC}
\scal{ Z,W}_{bc}  = ZW^{*} = \alpha \overline{\alpha'} \eu  + \beta \overline{\beta'}\es  ,
\end{align}
for $ Z= \alpha \eu  + \beta \es $ and $W=\alpha' \eu  + \beta' \es $ in  $\BC$, so that the induced bicomplex norm coincides with the usual Euclidean norm in $\mathbb{R}^{4}$ given by the modulus
\begin{align}\label{euclideannorm}
|Z|^{2}_{bc}=|z_{1}|^{2}+|z_{2}|^{2}= \frac{1}{2} \left(  |\alpha|^{2}+|\beta|^{2}\right) 
\end{align}
for given $Z=z_1 + z_2 j=\alpha \eu + \beta \es $; $z_1, z_2,\alpha,\beta \in \C=:\C_i$. We next perform the space of all $\BC$--valued measurable functions $f$ on $\BC$ subject to $\norm{f}_{bc} <+\infty$. Here $\norm{f}_{bc}$ is the bicomplex norm induced from the bicomplex inner product
 $$ \scal{ f,g}_{bc}:= \int_{\BC} \scal{ f(Z),g(Z)}_{bc} e^{-\nu |Z|^{2}}  d\lambda(Z) $$
 by means of \eqref{inducednorm}, $d\lambda(Z) $ being the Lebesgue measure on $\R^4$. 
More explicitly, we have
\begin{align}\label{bcsp}
\scal{f,g}_{bc} = \scal{f_1,g_1}_{L^{2,\nu}(\C^2)} \eu + \scal{f_2,g_2}_{L^{2,\nu}(\C^2)} \es
\end{align}
for every $f=f_1 \eu + f_2\es  ,g= g_1 \eu +  g_2 \es$. The functions
$f_k$ and $g_k$, for $k=1,2$, are seen as $\C$--valued functions on $ \C^2 $ in the variables $(z_1,z_2)$. Thus, the following decomposition 
\begin{align} \label{DecompHilbert1}
\Hbc = L^{2,\nu}(\C^2) \eu  +  L^{2,\nu}(\C^2) \es,
\end{align}
holds true thanks to 
\begin{align}\label{normL2}
\norm{f}_{\Hbc}^{2}= \frac{1}{2}  \left( \norm{f_1}_{L^{2,\nu}(\C^2)}^{2} + \norm{f_2}_{L^{2,\nu}(\C^2)}^{2} \right) .
\end{align}
Another interesting decomposition of the infinite bicomplex Hilbert space $\Hbc$ 
with respect to the idempotent representation states that 
	\begin{align} \label{DecompHilbert2}
\Hbc = L^{2,\frac{\nu}{2}}(\C^2) \eu  +  L^{2,\frac{\nu}{2}}(\C^2) \es .
\end{align}
Succinctly, for every $f\in \Hbc$, there exist $\phi^{\pm} \in L^{2,\frac{\nu}{2}}(\C^2)$ such that
	$$f(\alpha \eu+\beta \es)=\phi^+(\alpha,\beta)\eu+\phi^-(\alpha,\beta)\es$$
		with 
	\begin{align} \label{norm-bc2}
	\norm{f}_{\Hbc}^2 &=  \frac{1}{2} \left( \norm{\phi^+}^{2}_{L^{2,\frac{\nu}{2}}(\C^2)} + \norm{\phi^-}^{2}_{L^{2,\frac{\nu}{2}}(\C^2)} \right).
	\end{align}
	
In the sequel, we will denote $\Hbc$ simply $L^2(\BC)$ when $\nu=0$.

	\section{Bicomplex frames}
	
	In the sequel, $\mathcal{H}_{bc}$ will denote a separable bicomplex Hilbert space with bicomplex inner product $\scal{\cdot,\cdot}_{bc}$, linear in the
first entry, and denotes by $\norm{\cdot}_{bc}$ the associated bicomplex norm, $\norm{\cdot}_{bc}^2 =|\scal{\cdot,\cdot}|$. Let $\{f_n, n=0,1,2,\cdots\}$ be a countable family in $\mathcal{H}_{bc}$. The different notions from the classical frame theory can be extended, in a natural way, to the bicomplex setting.

  \begin{defn} 
 	The sequence $(f_n)_n$ is said to be a bicomplex basis for $\mathcal{H}_{bc}$ if for every $f\in H$ there exists a unique sequence of bicomplex numbers $(c_n)_n$ such that 
 	$$ f = \sum_{n=0}^\infty c_n f_n .$$
 	It is said to be  a bicomplex orthonormal basis if in addition $(f_n)_n$ is an orthonormal set. 
 \end{defn}
 
 \begin{defn}  A basis $(f_n)_n$ is said to be a bicomplex bounded basis if it satisfies the condition
 	$$0 < \inf |\scal{f_n,f_n}|_{bc} \leq  sup |\scal{f_n,f_n}|_{bc}< +\infty.$$
 	It is unconditional if for every
 	$f\in H$, the corresponding series $\displaystyle f=\sum_n c_n(f) f_n$; $c_n(f)\in \BC$,
 	converges for every rearrangement of its terms.
 \end{defn}

   \begin{prop}\label{Propbub}  If $(f_n)_n$; $f_n=f_n^+\eu + f_n^-\es$, is a bicomplex bounded unconditional basis for $L^2(\BC)$, then $(f_n^+)_n$ or $(f_n^-)_n$ is a bounded unconditional basis for $L^2(\C^2)$.
\end{prop}

\begin{proof} 
  Using the idempotent decomposition, the bc-basis property, and the bc-unconditionality of $(f_n)_n$ in $L^2(\BC)$ are clearly equivalent to that $(f_n^+)_n$ and $(f_n^-)_n$ being unconditional bases for $L^2(\C^2)$. We need only to prove the boundedness property. Indeed, starting from \eqref{normL2},
   we obtain 
$$ 
sup |\scal{f_n,f_n}|_{bc} \leq \frac 1{\sqrt2} \left( \sup  \norm{f_n^+} + \sup \norm{f_n^-} \right)  .
$$
Therefore, $sup |\scal{f_n,f_n}|_{bc}$ is finite if and only if $\sup  \norm{f_n^+}$ and $\sup \norm{f_n^-}$ are finites.
In a similar way, we can prove the following
$$\frac 1{\sqrt2} \sup\left( \inf   \norm{f_n^+} , \inf  \norm{f_n^-}\right) \leq inf |\scal{f_n,f_n}|_{bc} 
 .$$                                                                          
	This proves $inf |\scal{f_n,f_n}|_{bc}$ is positive if and only if $\inf   \norm{f_n^+}$ or $\inf   \norm{f_n^-} $ is positive. This completes the proof.
\end{proof}

According to the previous proof, it is evident to see that the converse of Proposition \ref{Propbub} holds true if both of the idempotent components $(f_n^+)_n$ and $(f_n^-)_n$ are assumed to be bounded unconditional basis for $L^2(\C^2)$. 
Moreover, the following assertion gives a strong version of the converse. 
 Its proof is contained in the previous one.

\begin{prop}\label{Propbubc}   
If	$(f_n^+)_n$ (or  $(f_n^-)_n$) is a bounded unconditional basis for $L^2(\C^2)$ and 
the sequences $(\norm{f_n^+} )_n$, $(\norm{f_n^-})_n$ are upper bounded, then $(f_n)_n$ is a bicomplex bounded unconditional basis for $L^2(\BC)$.
\end{prop}

In analogy with the standard case, we propose the following definitions for bc-Reisz bases and bc-frames.
 
 \begin{defn} A sequence $(f_n)_n$ is a bicompex Riesz basis for $\mathcal{H}_{bc}$ if $\overline{span\{f_n; n \}} = \mathcal{H}_{bc}$ and
 	there exist $A, B > 0$ such that 
 	$$ A \sum_{n=0}^\infty |c_n|^2_{bc} \leq \norm{\sum_{n=0}^\infty c_n f_n}^2_{bc} \leq B \sum_{n=0}^\infty |c_n|^2_{bc} $$
 	for all $\displaystyle f=\sum_n c_n f_n  \in H$.
 \end{defn}
 
 \begin{defn} The sequence $(f_n)_n$ is called a bc-frame for $\mathcal{H}_{bc}$ if there are two constants $0 < A \leq B $ (frame bounds) such  that for every $f \in H$, we have
 	$$ A \norm{f}^2_{bc} \leq \sum_{n=0}^\infty |\scal{f,f_n}|^2_{bc} \leq B  \norm{f}^2_{bc} .$$
 	A bc-frame is said to be
 	tight if in addition $A=B$ and a Parseval bc-frame if $ A=B=1$. 
 	It is called exact if it ceases to be a bc-frame whenever any single element is deleted from it.   
 \end{defn}

The following example is specific for the bc-frames and is generated by a given orthonormal basis $e_n$ of $L^2(\C^2)$.

  	\begin{example}\label{Fexple} Let $(c_n)_n$ be bicomplex sequence such that $c_n=a_n\eu +b_n\es$; $a_n,b_n\in\C$, and consider the set of functions
  		$$e_{m,n}(z_1+jz_2):= a_ne_m(z_1,z_2) \eu + b_me_n(z_1,z_2)\es .$$
  	If $\sum\limits_{n=0}^\infty |a_n|^2=a$ and $\sum\limits_{n=0}^\infty |b_n|^2=b$ converge, then $e_{m,n}$, for varying $m$ and $n$, is a bc-frame for $L^2(\BC)$. 
  Indeed, by means of \eqref{bcsp}, we get 
  $\scal{ e_{m,n}, e_{j,k}}_{bc}=  a_n \overline{a_k} \delta_{m,j} \eu + b_m \overline{b_j} \delta_{n,k} \es$.
  This shows that $e_{m,n}$ is not orthogonal nor necessary normalized. However, direct computation shows that 
  $$ \sum_{m,n=0}^\infty |\scal{f,e_{m,n}}|^2_{bc} = \frac 12 \left( a \norm{f^+}^2_{L^2(\C^2)} +  b \norm{f^-}^2_{L^2(\C^2)} \right) .$$ 
  Therefore, 
  $$ \min(a,b) \norm{f}^2_{bc}\leq  \sum_{m,n=0}^\infty |\scal{f,e_{m,n}}|^2_{bc}  \leq \max(a,b) \norm{f}^2_{bc}.$$
  By specifying $(a_n)_n$ and $(b_n)_n$, we get tight and Parseval bc-frames.
\end{example}

The next result is a fundamental tool in our expository and characterize the bc-frames in terms of the classical ones. 

  	\begin{theorem}\label{thmchar}
    	The sequence $(f_n)_n$ is a bc-frame for $L^2(\BC)$ with best frame bounds $A$ and $B$ if and only if their components $(f_n^\pm)_n$, $f_n=f_n^+\eu+f_n^-\es$, are frames for $L^2(\C^2)$ with best frame bounds $a^\pm$ and $b^\pm$. Moreover, we have $A=\min\{a^+,a^-\}$ and $B=\max\{b^+,b^-\}$. 
  \end{theorem}
  
  \begin{proof}
  	Let  $f_n,f\in L^2(\BC)$ and write them as $f_n=f_n^+\eu+f_n^-\es$ and $f=f^+\eu+f^-\es$ with $f^+,f^-,f_n^+,f_n^-\in L^2(\C^2)$. 
  	Then, in virtue of \eqref{normL2},   
the condition 
  	 $$ A \norm{f}^2 \leq \sum_{n=0}^\infty |\scal{f,f_n}|^2 \leq B  \norm{f}^2 $$
  	 for $f \in L^2(\BC)$ is equivalent to
  	 \begin{align}\label{eq}
  	  A &\left( \norm{f^+}^2_{L^2(\C^2)} + \norm{f^-}^2_{L^2(\C^2)}\right) \nonumber \\ &\qquad\qquad\qquad \leq \sum_{n=0}^\infty \left( \left|\scal{f^+,f_n^+}_{L^2(\C^2)}\right|^2 + \left|\scal{f^-,f_n^-}_{L^2(\C^2)}\right|^2 \right) 
  	  \\& \qquad\qquad\qquad \qquad \qquad \qquad \qquad \qquad  \leq B\left( \norm{f^+}^2_{L^2(\C^2)} + \norm{f^-}^2_{L^2(\C^2)}\right)  \nonumber
  	 \end{align}  	
  for every $f^+,f^-\in L^2(\C^2)$. 
  Accordingly, if $(f_n)_n$ is a bc-frame for $L^2(\BC)$ with best frame bounds $A$ and $B$ then the components $(f_n^+)_n$ (resp. $(f_n^-)_n$) is a frame for $L^2(\C^2)$ by taking $f^-=0$ (resp. $f^+=0$) in \eqref{eq}. Their best frame bounds $a^+,b^+$ (resp. $a^-,b^-$) satisfy $A\leq\min\{a^+,a^-\}$ and $B\geq\max\{b^+,b^-\}$. 
  Conversely, if $(f_n^+)_n$ and $(f_n^-)_n$ are frames for $L^2(\C^2)$ with best frame bounds $a^+,b^+$ and $a^-,b^-$, respectively, then  $(f_n)_n$ is a bc-frame for $L^2(\BC)$ with best frame bounds $A$ and $B$ satisfying $A\geq\min\{a^+,a^-\}$ and $B\leq\max\{b^+,b^-\}$. 
  This completes the proof of that $(f_n)_n$ is a bc-frame for $L^2(\BC)$  if and only $(f_n^+)_n$ and $(f_n^-)_n$ are frames for $L^2(\C^2)$ with $A=\min\{a^+,a^-\}$ and $B=\max\{b^+,b^-\}$. 
  \end{proof}

\begin{remark}
	The assertion of Proposition \ref{Propbub}  and Theorem \ref{thmchar} remain valid for bc-frames $(f_n)_n$ for general bicomplex Hilbert space $\mathcal{H}_{bc}=\mathcal{H}^+\eu + \mathcal{H}^-\es$, with $ f=f^+\eu + f^-\es$ and $f^\pm \in \mathcal{H}^\pm$.
\end{remark}
 
   	\begin{example}\label{Fexple2} Let $(c_n)_n$ a bicomplex sequence as in Example \ref{Fexple}. Let $(f_n)_n$, with $f_n= f_n^+\eu + f_n^- \es$, be a bc-frame for $L^2(\BC)$. Then, the set of functions
 	$$f_{m,n}(z_1+jz_2):= a_nf_m^+(z_1,z_2) \eu + b_mf_n^-(z_1,z_2)\es $$
 	defines a new  bc-frame for $L^2(\BC)$. This immediately follows from Theorem \ref{thmchar} and generalizes Example  \ref{Fexple}.
 \end{example}

   	\begin{remark}\label{} Theorem \ref{thmchar} and Example \ref{Fexple2} show why bc-frames may be of interest and that they contain the classical ones as particular subclasses. In fact, since $L^2(\C^2)$ can be embedded in a natural way in $L^2(\BC)$, any frame $(h_n)_n$ for $L^2(\C^2)$ can be seen as a bc-frame for  $L^2(\BC)$ by considering $h_n \eu + h_n \es = h_n$. 
\end{remark}

  According to Theorem \ref{thmchar}, a number of important properties for bc-frames are described forthwith thanks to the previous characterization. For example, it is well known that classical Riesz basis can be characterized as the data of a sequence $(f_n)_n$ in $\mathcal{H}_{bc}$ which is the image of an orthonormal basis under a bounded invertible linear operator. 
  We claim that this remains valid for bicomplex Riesz basis. This readily follows from the following fact 
  (whose proof is similar to the one provided to Theorem \ref{thmchar}). 
  
    \begin{prop}\label{PropRiesz} The sequence $(f_n)_n$ is a bc-Riesz basis if and only the idempotent components $(f_n^\pm)_n$ are Riesz bases.
     \end{prop}

  Moreover, we have the following characterization of completeness  of bc-frame $(f_n)_{n}$.
  
  \begin{prop}\label{PropParseval}
  	The bc-frame $(f_n)_{n}$ is complete in $L^2(\BC)$ if and only if $(f_n^+)_{n}$ and $(f_n^-)_{n}$ are both complete in $L^2(\C^2)$.
  	The same observation holds true for the Parseval bc-frames. 
  \end{prop}
  
  However, one has to be careful as shown by the following assertions.
   
  \begin{prop}\label{Proptight} If $(f_n)_n$ is a tight bc-frame for $L^2(\BC)$, then 
  $(f_n^+)_n$ and $(f_n^-)_n$ are tight frames for $L^2(\C^2)$.
 \end{prop}

\begin{proof}
	The proof is straightforward using the direct implication in Theorem \ref{thmchar}.
	\end{proof}

\begin{remark}
   The converse is not in general true unless $(f_n^+)_n$ and $(f_n^-)_n$ are tight frames with the same best frame bounds. 
   \end{remark}
    
 \begin{prop}\label{Propexact}
  	Let $(f_n)_n$ be a bc-frame for $L^2(\BC)$ and assume that
  $(f_n^+)_n$ (or $(f_n^-)_n$) is an exact frame for $L^2(\C^2)$. Then $(f_n)_n$ is an exact bc-frame for $L^2(\BC)$. 
    \end{prop}

\begin{proof}  
	  The non-exactness of $(f_n)_n$ is equivalent to the existence of some $u$ such that $(f_n)_{n\ne u}$ is still a bc-frame. By Theorem \ref{thmchar}, this is equivalent to the sets $(f_n^+)_{n\ne u}$ and $(f_n^-)_{n\ne u}$ be frames for $L^2(\C^2)$. This means that $(f_n^+)_{n}$ and $(f_n^-)_{n}$ are both not exact (at last at $u$).  This completes the proof.
\end{proof}

	From this, there is non reason to have the converse of Proposition \ref{Propexact} "$(f_n)_{n}$ is exact for $L^2(\BC)$ implies $(f_n^+)_{n}$ or $(f_n^-)_{n}$ is an exact frame for $L^2(\C^2)$". This is not true in general as shown by the following counterexample. 

\begin{cexample}\label{CExp} Let $(\varphi_n)_n$ and $(\psi_n)_n$ be two frames for $L^2(\C^2)$ such that $(\varphi_n)_{n\ne u}$ and $(\psi_n)_{n\ne v}$ are exact with $u\ne v$. We assume that $u$ (resp. $v$) is only value satisfying this property.
	Example of such frame exists and one can consider $(\varphi_n)_n=\{ e_0\} \cup (e_n)_{n\geq 0}$, where $(e_n)_n$ is an orthonormal basis. 
	We next perform $f_n = \varphi_n \eu + \psi_n\es$ which  clearly is a bc-frame for $L^2(\BC)$ (by Theorem \ref{thmchar}). Moreover, it is exact, i.e.,  
	$(f_n)_{n\ne m}$ ceases to be a bc-frame for any arbitrary $m$. To this end, notice that we necessary have $m\ne u$ or $m \ne v$. For the first case for example ($m\ne u$),  
	the first component $(f_n^+)_{n\ne m} = (\varphi_n)_{n\ne m}$ is not a frame for $L^2(\C^2)$ by assumption  ($(\varphi_n)_n$ is exact at $u$ only), and therefore one concludes for $(f_n)_{n\ne m}$ by making again use of Theorem \ref{thmchar}. 
\end{cexample}

	If we denote by $N_{Exact}((f_n)_n)$ the set of indices $k$ for which $(f_n)_{n\ne k}$ is not a frame, 
	$$  N_{Exact}((f_n)_n) =\{k; \, (f_n)_{n\ne k} \mbox{ is a frame} \},$$
	then the exactness of $(f_n)_n$ becomes equivalent to  $N_{Exact}((f_n)_n)=\emptyset$. Subsequently, 
		\begin{align}\label{idNexact}  N_{Exact}((f_n)_n) =  N_{Exact}((f_n^+)_n) \cap  N_{Exact}((f_n^-)_n) 
		\end{align}
		by means of Theorem \ref{thmchar} and the definition of $N_{Exact}((f_n)_n)$. Hence, Proposition \ref{Propexact} appears as particular case of the following 
	
\begin{theorem}
	 A bc-frame  $(f_n)_n$ for $L^2(\BC)$ is exact if and only if 
	 $$  N_{Exact}((f_n^+)_n) \cap  N_{Exact}((f_n^-)_n)=\emptyset.$$ 
\end{theorem}

\begin{remark} The assertion in the counterexample \ref{CExp} can be reproved easily making use of $  N_{Exact}((f_n)_n)$. Indeed, since $N_{Exact}((\varphi_n)_n)=\{u\}$ and $N_{Exact}((\psi_n)_n)=\{v\}$, we get and  
	  $N_{Exact}((\varphi_n\eu + \psi_n\es)_n) =\emptyset$ by means of \eqref{idNexact}, and therefore $\varphi_n\eu + \psi_n\es$ is an exact bc-frame for $L^2(\BC)$.
  \end{remark}
   
 Accordingly, one proves that the well-known fact that "a frame for $L^2(\R^d)$ is exact if and only if it is a Riesz basis (see e.g. \cite[Theorem 7.1.1, p. 166]{Christensen2016})" is no longer valid for bc-frames. However, we assert the following
 
  \begin{prop}\label{PropRiesz}
  	 If $(f_n)_{n}$ is a bc-Riesz basis for $L^2(\BC)$, then $(f_n)_{n}$ is exact bc-frame for $L^2(\BC)$. 
  	   \end{prop}
   
\begin{proof} This is immediate making use of the fact that  $(f_n)_{n}$ is a Riesz basis for $L^2(\BC)$ if and only if the component sequences $(f_n^\pm)_{n}$ are Riesz bases for $L^2(\C^2)$ combined with that the exactness of ordinary frames for $L^2(\R^d)$ is equivalent to being Riesz bases \cite[Theorem 7.1.1, p. 166]{Christensen2016} and Theorem \ref{thmchar}. 
	  \end{proof} 

For complex Hilbert spaces all bounded unconditional bases are equivalent
to an orthonormal basis. More precisely, if $(\psi_n)_n$ is a bounded unconditional basis for $\mathcal{H}$, then there exists an orthonormal basis $(e_n)_n$ and a bounded invertible operator $U :\mathcal{H} \longrightarrow \mathcal{H}$ such that $\psi_n =
Ue_n$ for each $n$. 
This result remains valid for bc-frames thanks to Proposition \ref{Propbub}. Moreover, there is a close relation between exact frames and bounded
unconditional bases.

\begin{theorem}[\cite{HeilWalnut1989,Young1980}]\label{exactBEUnc}	 
	A frame $(f_n)_n$ for a complex Hilbert space $\mathcal{H}$ is exact
	if and only if it is a bounded unconditional basis.
\end{theorem}

For bc-frames, we assert the following 

\begin{prop}\label{Propbubexact}  If $(f_n)_n$ is a bicomplex bounded unconditional basis for $\mathcal{H}_{bc}$, then $(f_n)_n$ is exact for $\mathcal{H}_{bc}$.
\end{prop}

\begin{proof}
	This follows from Theorem \ref{exactBEUnc} combined with Propositions \ref{Propbub} and \ref{Propexact}.
\end{proof}

\section{bc-frame operator and Weyl--Heisenberg bc-frames}
\subsection{bc-frame operator.}

Let $(\mathcal{H}_{bc}, \scal{,}_{bc})$ be a functional bc-Hilbert space on $\BC$  endowed with the bc-scaler product 
$$ \scal{f,g}_{bc} =\int_{\BC} f(Z) [g(Z)]^* d\lambda(Z).$$
Thus, we decompose $\mathcal{H}_{bc}$ idempotentically as $\mathcal{H}_{bc}= \mathcal{H}^+\eu +\mathcal{H}^-\es$ where $\mathcal{H}^\pm$ are $\C$-Hilbert spaces on $\C$. For given bc-bicomplex frame $(f_n)_n$ for $\mathcal{H}_{bc}$, the components $(f^\pm_n)_n$ are ordinary frames for $\mathcal{H}^\pm$.
Therefore, one may define the analysis operators $T_\pm : \mathcal{H}^{\pm} \longrightarrow \ell^2_{\C}$ and their adjoint $T_\pm^{adj} : \ell^2_{\C} \longrightarrow \mathcal{H}^{\pm}$ by 
$$ T_\pm f^\pm = \left( \scal{f^\pm,f^\pm_n}\right)_n \quad \mbox{and} \quad T_\pm^{adj}((c_n)_n) = \sum_{n=0}^\infty c_n f^\pm_n.$$ 
We can define the bc-analysis operator $ T :  \mathcal{H}_{bc} \longrightarrow \ell^2_{\BC}$ for the bicomplex Hilbert space $\mathcal{H}_{bc}$ by 
$$ Tf := \left( T_+\eu + T_-\eu  \right) f =  T_+f^+\eu + T_-f^-\eu  = (\scal{f,f_n}_{bc})_n $$
with $f=f^+ \eu + f^- \eu \in \mathcal{H}_{bc}$. 
Its adjoint  $ T^{bc-adj} :  \ell_{\BC}  \longrightarrow \mathcal{H}_{bc}$ with respect to bicomplex hilbertian structure is shown to be given by
$$ T^{bc-adj} f 
=  T_+^{adj}f^+\eu + T_-^{adj}f^-\eu .$$ 
 Therefore, we define the bc-frame operator $S: \mathcal{H}_{bc} \longrightarrow \mathcal{H}_{bc}$ to be 
$$ S f = \left( T_+\eu + T_-\eu  \right)^{bc-adj} \left(  T_+\eu + T_-\eu  \right)= S_+ f^+ \eu + S_- f^- \es $$
with $S_\pm=T_\pm^{adj}T_\pm$ are the classical frame operators associated to $(f^\pm_n)_n$ for $\mathcal{H}^{\pm}$. By construction, the operator $S$ inherits from $S_{\pm}$ their basic properties. Notice for instance that we have following   
$$ Sf =\sum_{n=0}^\infty \scal{f,f_n}_{bc} f_n ,$$
$$ \scal{Sf,f}_{bc}= \sum_{n=0}^\infty |\scal{f,f_n}_{bc}|^2 $$
and 
$$ f =\sum_{n=0}^\infty \scal{f,S^{-1}f_n}_{bc} f_n .$$
Moreover, $S$ is clearly invertible, self-adjoint $\scal{Sf,g}_{bc}=\scal{f,Sg}_{bc}$ and bounded operator. Its norm satisfies the following estimation
 $$\norm{S}^2_{op} \leq \max\left( \norm{S_+}^2_{op}, \norm{S_-}^2_{op}\right)  = \max\left( {b^+}^2_{opt}, {b^-}_{opt}\right) .$$ 
Moreover, $S$ is hyperbolic positive in the sense that $\scal{Sf,f}_{bc} \in \mathbb{D}^+$ for every $f\in \mathcal{H}_{bc}$, where $\mathbb{D}^+= \R^+ \eu + \R^+\eu$ denotes the set of positive hyperbolic numbers $\mathbb{D}^+= \R^+ \eu + \R^+\eu$. 

Although, the operator frame is naturally extended to the bicomplex context, we will be careful when examining their properties. This is closely connected to materials discussed in the previous section. In fact, from Proposition \ref{Proptight}, we know that tightness of a bc-frame $(f_n)_n$ for $\mathcal{H}_{bc}$ implies the one of the frames $(f^\pm_n)_n$ for $\mathcal{H}^{\pm}$. Thus, by means of \cite[Proposition 5.1.1., p. 86]{Gr\"ochenig2001}, we have $S_\pm = a^\pm_{opt} Id_{\mathcal{H}^{\pm}}$; $a^\pm_{opt}=b^\pm_{opt} \in \R^+$, and therefore 
$$ S   = a^+_{opt} Id_{H^+}\eu + a^-_{opt} Id_{H^-}\eu = \left(  a^+_{opt} \eu + a^-_{opt}\eu \right)  Id_{\mathcal{H}_{bc}}.$$
This proves the following assertion  

\begin{prop}
	If $(f_n)_n$ is a tight bc-frame for $\mathcal{H}_{bc}$, then $S= d Id_{\mathcal{H}_{bc}}$ for certain positive hyperbolic number $d\in \mathbb{D}^+$. 	
\end{prop}	

\begin{remark}
	The converse in not valid in general unless we assume that $d\in\R^+$ (i.e., $a^+_{opt} = a^-_{opt}$). Thus for complex-valued bc-frame, we recover the classical result characterizing tight frames as the operator $d Id_{\mathcal{H}_{bc}}$.
\end{remark}

\subsection{Weyl-Heisenberg bc-frames.}

The so-called Weyl--Heisenberg (or Gabor) frames are the famous frames for $L^2(\R)$ that can be generated from a single element (called mother wavelet). More exactly, they are frames of functions of the form
$$ \mathcal{G}(a,b,g) := \{W_{na,mb} (g) (t) = e^{ i m bt} g(t - na)  ; \, m, n \in \Z\},$$
 where $a, b > 0$ and $g \in L^2(\R)$ are fixed, and $W_{a,b} $ denotes the Weyl operator 
$$W_{a,b} (g) (t) :=  e^{ ibt} g(t - a) = M_{b} T_{a} g(t) .$$.  
The frameness of Weyl--Heisenberg systems $\mathcal{G}(a,b,g)$ in $L^2_{\C}(\R)$ has been
extensively discussed in several papers. 
See for instance \cite{Daubechies1990,HeilWalnut1989,Christensen2016,Deng1997,Gr\"ochenig2001}.
The next result is an example of assertions providing us with sufficient conditions on $g$ and the lattice parameters $a$ and $b$ to $\mathcal{G}(a,b,g) $ be a frame. Namely,

\begin{theorem}[\cite{HeilWalnut1989}] \label{WHframe}	
	Let $g$ be a compactly supported function with support contained in some interval $I$ of length $1/b$ and such that 
	$$ \alpha \leq \sum_{n=0}^\infty |T_{na} (g) (t)|^2 \leq \beta$$
	almost everywhere on $\R$ for some constants $\alpha,\beta>0$. Then, 
	the Weyl--Heisenberg system $\mathcal{G}(a,b,g) $ is a frame for $L^2(\R)$ with frame bounds
	$\alpha/b$ and $\beta/b$.
\end{theorem}

The introduction of bicomplex analogue of Weyl--Heisenberg (W-H) systems  can be accomplished in different ways. In the sequel, we confine our attention to the natural ones. By considering two classical W-H systems $\mathcal{G}(a,b,g)$ and $\mathcal{G}(c,d,h)$ for  $L^2(\R)$, we perform the following 
$$  \mathcal{G}^{\nu,\mu}(\Gamma_{hyp}(A,B), f)  := \{W^{\nu,\mu}_{nA,mB} f  ; m,n\in \Z\}
$$
 associated to  $f=g\eu+h\es$, the hyperbolic lattice $\Gamma_{hyp}(A,B)= \Z A + \Z B$; $A=a\eu+c\es, B=b\eu+d\es\in \D^+$ and the modified Weyl operator \cite{ElGhZ2020}
$$ W^{\nu,\mu}_{nA,mB} f (t) :=  
[W^\nu_{na,mb} g] (t) \eu + [W^\mu_{nc,md} h] (t) $$
defined as  as projective representation of two copies of  $ W^\nu_{a,b} g(t) = e^{\nu  bt}g(t-a)$. Here $\nu^2=\mu^2=-1$. 
By taking $a=b$, $c=d$ and $g=h$ we recover the classical notion of W-H system.

\begin{defn}
	We call $\mathcal{G}^{\nu,\mu}(\Gamma_{hyp}(A,B), f)$ a bicomplex W-H system for the Hilbert space $L^2_{\BC}(\R)$.
	\end{defn}

Therefore, by means of Theorems \ref{thmchar} and \ref{WHframe}, we deduce easily the following result for W-H  bc-systems in the bicomplex Hilbert space of bicomplex-valued functions $f$ on $\R$ such that $\norm{f}_{bc}<+infty$.

\begin{prop}  $\mathcal{G}^{\nu,\mu}(\Gamma_{hyp}(A,B), f)$ generates a bc-frame for $L^2_{\BC}(\R)$ under the assumption that $f$ are compactly supported with support contained in some interval $I$ of length $\min(1/b,1/d)$ and 
$$ \alpha' \leq \sum_{n=0}^\infty |(T_{na} \eu + T_{nc} \es) (f) (t)|^2_{bc} \leq \beta'$$
almost everywhere on $\R$ for certain constants $\alpha'$ and $\beta'$.
\end{prop}
	
	Thus, the existence of  W-H bc-frames for $L^2_{\BC}(\R)$ requires in particular $ab\leq 1$ or $cd\leq 1$ by Theorem \ref{thmchar} and \cite[Corollary 7.5.1., p. 138]{Gr\"ochenig2001}. Moreover, the most properties valid for complex W-H frames can be established easily for the W-H bc-frames. 
Notice then even the critical case $ab=1$ characterize the exact W-H frame for $L^2(\R)$, the exactness  of  $\mathcal{G}^{\nu,\mu}(\Gamma_{hyp}(A,B), f)$ is not characterized by  
$AB=1$ in $\D^+$, i.e., $ab=1$ and $cd=1$. Namely, the following assertion readily follows using \cite[Corollary 7.5.2. p. 139]{Gr\"ochenig2001} and Proposition \ref{Propexact}

\begin{prop}\label{propWHab}
If $\mathcal{G}^{\nu,\mu}(\Gamma_{hyp}(A,B),f)$ is a Weyl--Heisenberg bc-frame for $L^2_{\BC}(\R)$ and $ab=1$ or $cd=1$, then it is exact. 
\end{prop}

\begin{remark}
	The converse of Proposition \ref{propWHab} is not true in general.
\end{remark}

The previous formalism for constructing W-H bc-frames for $L^2_{\BC}(\R)$ can be extended (in a similar way), thanks to Theorem \ref{thmchar}), to the Hilbert space $L^2_{\BC}(\D)$ on hyperbolic numbers by considering the family of functions
\begin{align*}
[W^\nu_{na,mb} \varphi ] (x\eu+y\es) \eu +  [W^\mu_{pc,qd} \psi] (x\eu+y\es) ,
\end{align*}
for varying 4-uplet $M=(m,n,p,q)\in \Z^4$. Here $a,b,c,d>0$ are fixed reals and $\varphi,\psi $ are fixed complex-valued functions in $L^2_{\C}(\R)$. 
The following result shows that it is not true in general that one can generate W-H bc-frames for  $L^2_{\BC}(\D)$  by means of theorem \ref{thmchar} and
 starting from given W-H frames for  $L^2_{\C}(\R)$. In fact, for fixed $g,h \in L^2_{\C}(\R)$,  we define 
\begin{align}\label{WHbcGen}
\psi_M^{\nu,\mu}(x\eu + y\es) := e^{-y^2/2}[W^\nu_{na,mb} g] (x) \eu + e^{-x^2/2} [W^\mu_{nc,md} h] (y).
\end{align} 
Thus, we can prove the following

\begin{prop} 
	Assume $\mathcal{G}(a,b,g)$ and $\mathcal{G}(c,d,h)$ are frames for $L^2_{\C}(\R)$. Then, the bicomplex W-H system $\psi_M^{\nu,\mu}$ in \eqref{WHbcGen} is not a bc-frame for 
	$L^2_{\BC}(\D)$.
\end{prop}

	\begin{proof}
By means of Theorem \ref{thmchar} the framness of $\psi_M^{\nu,\mu}$, in $L^2_{\BC}(\D)$, is equivalent to the framness of their components $e^{-y^2/2}[W^\nu_{na,mb} g] (x)$ and $e^{-x^2/2} [W^\mu_{nc,md} h] (y)$ in $L^2_{\C}(\D)$. It should be noticed here that the tensor product $h_u\oplus h_v (x,y) = h_u(x)h_v(y)$ of Hermite functions is an orthogonal basis of 
$$L^2_{\C}(\D) =\left\{ \psi : \D \longrightarrow \C; \int_{\D} \psi(\xi) \overline{\psi(\xi)} d\lambda(\xi) \right\},$$ and that direct computation using Fubini theorem shows that we have 
\begin{align*}
\sum_{m\in\Z}\sum_{n\in\Z} \left|\scal{h_u\otimes h_v ,  W^\nu_{na,mb} g \otimes h_0 }_{L^2_{\C}(\D)} \right|^2
&= \pi \sum_{m\in\Z}\sum_{n\in\Z}  \left|\scal{ h_u  ,W^\nu_{na,mb} g}_{L^2_{\C}(\R)} \right|^2   \delta_{v,0} .
\end{align*}
This shows in particular that the considered sequence $W^\nu_{na,mb} g \otimes h_0$ is not a frame in $L^2_{\C}(\D)$. This completes the proof.
	\end{proof}

However, we have
\begin{prop} 
	If $(a,b,g)$ and $(c,d,h)$ generate Bessel sequence of W-H type in $L^2_{\C}(\R)$, then the family $\psi_M^{\nu,\mu}$ in \eqref{WHbcGen}, for varying $M\in\Z^4$, is a Bessel sequence in $L^2_{\BC}(\D)$.
\end{prop}

\begin{proof}
	By similar arguments as in the proof of Theorem \ref{thmchar}, we need only to prove 
	the components $e^{-y^2/2}[W^\nu_{na,mb} g] (x)$ and $e^{-x^2/2} [W^\mu_{nc,md} h] (y)$ are Bessel sequences in $L^2_{\C}(\D)$. Thus, for arbitrary $\Phi\in L^2_{\C}(\D)$, the partial function $x\longmapsto \Phi_y(x):=\Phi(x\eu+y\es)$ is clearly in $L^2_{\C}(\R)$. Moreover, 
	\begin{align*}
	\scal{\Phi, e^{-y^2/2} W^\nu_{na,mb} g }_{L^2_{\C}(\D)} 
	&= \int_{\D} \Phi(x\eu+y\es)  \left( e^{-y^2/2} W^\nu_{na,mb} g (x) \right)^{*} dxdy 
	\\&= \int_{\R}  e^{-y^2/2}  \scal{\Phi_y, W^\nu_{na,mb}g}_{L^2_{\C}(\R)} dy.
	\end{align*}
	Now, since $g$ generates a Bessel sequence in $L^2_{\C}(\R)$, we see that
	\begin{align*}
	\sum_{m\in\Z}\sum_{n\in\Z} \left|\scal{\Phi, e^{-y^2/2} W^\nu_{na,mb} g }_{L^2_{\C}(\D)} \right|^2 
	&\leq  \int_{\R}  e^{-y^2}  \sum_{m\in\Z}\sum_{n\in\Z}   \left|\scal{\Phi_y, W^\nu_{na,mb}g}_{L^2_{\C}(\R)}\right|^2 dy
	\\&\leq  \int_{\R}  e^{-y^2} \norm{\Phi_y}^{2}_{L^2_{\C}(\R)}  dy
\\&\leq c \norm{\Phi_y}^{2}_{L^2_{\C}(\D)}  
	\end{align*}
	for some constant $c$. This shows that $e^{-y^2/2} W^\nu_{na,mb} g $ is a Bessel sequence in $L^2_{\C}(\D)$.  	
\end{proof}

\section{Concluding remarks}
In this note, we have presented a natural extension of some notions from classical frame theory including operator frame and Weyl-Heisenberg frames to the bicomplex setting. We have briefly discussing their similarities and the differences to classical ones. The mean feature of bc-frame is Theorem \ref{thmchar}. The complete description of bc-frames needs further investigations. 
 However, the bc-hilbertian structure allows the consideration of non-trivial extensions for the existence of three complex conjugates and the divisibility by zero. Notice for instance that we define a $\dagger$-bc-frame to be a bicomplex sequence $(f_n)_n$ in a given infinite bc-Hilbert space $\mathcal{H}_{bc}$ if there exist $A,B>0$ such that 
$$ A\norm{f}_{bc}^2   \leq  \sum_{n=0}^\infty |\scal{f,f_n}_{bc}^\dagger  \scal{f,f_n}_{bc}^* | \leq B\norm{f}_{bc}^2$$
holds true  for every $f\in\mathcal{H}_{bc}$,
so that one recovers the classical definition if we restrict ourself to complex valued functional Hilbert spaces.  
We claim that this class possesses several interesting and surprising results that deserve spacial study. We hope to study this in detail in a forthcoming paper.

We conclude by noticing that other constructions of W-H bc-frames and bicomplex Wilson bases for bicomplex Bargmann space in \cite{ElGhZ2020} can be considered by benefiting from the rich structure of bicomplex Hilbert space, including the one the hyperbolic numbers and those arising from 
discretization of bicomplex Fourier--Wigner transform in \cite{ElGhZ2020} and associated to the bicomplex projective representations considered in \cite{ElGhZ2020}.

\quad 

\noindent{\bf Acknowledgments:}
The authors are grateful to Prof. Kabbaj Samir for inspiring discussions.
They also thank the members of Ahmed Intissar Seminar for assistance and fruitful discussions.

 \end{document}